\documentclass[review,ams]{elsarticle}

\usepackage{soul,color}
\usepackage{xcolor,colortbl}
\usepackage{graphicx}
\usepackage{amsmath}
\usepackage{latexsym}
\usepackage{subfigure}
\usepackage{relsize}
\usepackage{float}
\usepackage{multirow}
\usepackage{amsthm}
\usepackage{amsmath}
\usepackage[mathscr]{euscript}
\usepackage{mathrsfs}
\usepackage[font=small,labelfont=bf]{caption}
\usepackage{amsmath, amsthm, amscd, amsfonts, amssymb}
\usepackage[top = 1in, bottom = 1in, left = .8in, right = .8in]{geometry}
\newtheorem{thm}{Theorem}

 \newtheorem{lem}[thm]{Lemma}
 \newtheorem{defn}[thm]{Definition}
 
\theoremstyle{definition}
\newtheorem{definition}{Example}[section]

\definecolor{Gray}{gray}{0.85}
\newcolumntype{a}{>{\columncolor{Gray}}c}
\begin{document}
\makeatletter
\def\ps@pprintTitle{%
   \let\@oddhead\@empty
   \let\@evenhead\@empty
   \let\@oddfoot\@empty
   \let\@evenfoot\@oddfoot
}
\makeatother
\begin{frontmatter}

\title{LDG method for solving spatial and temporal fractional nonlinear convection-diffusion equations}

\address[add1]{Department of Applied Mathematics, Faculty  of Mathematical Sciences, Shahid Rajaee Teacher Training University, Tehran, Iran}

\author[add1]{Majid~Rajabzadeh}

\author[add2]{Moein Khalighi\corref{cor1}}
\ead{moein.khalighi@utu.fi}
\address[add2]{Department of Computing, Faculty of Technology, University of Turku, Finland}

\cortext[cor1]{Corresponding author}

\begin{abstract}
This paper focuses on a nonlinear convection-diffusion equation with space and time-fractional Laplacian operators of orders $1<\beta<2$ and $0<\alpha\leq1$, respectively. We develop local discontinuous Galerkin methods, including Legendre basis functions, for a solution to this class of fractional diffusion problem, and prove stability and optimal order of convergence $O(h^{k+1}+(\Delta t)^{1+\frac{p}{2}}+p^2)$. This technique turns the equation into a system of first-order equations and approximates the solution by selecting the appropriate basis functions. Regarding accuracy and stability, the basis functions greatly improve the method. According to the numerical results, the proposed scheme performs efficiently and accurately in various conditions and meets the optimal order of convergence. 
\end{abstract}

\end{frontmatter}
\section*{Introduction}\label{sec:intro}
In recent years, fractional differential equations have gained popularity among researchers due to their flexibility in science and engineering, which provides more degrees of freedom for integrodifferential equations in modeling various phenomena, such as optimal control problems \cite{soli1, soli2}, complex networks~\cite{PhysRevE.95.022409,10.1371/journal.pone.0154983}, and viscoelastic systems~\cite{matlob2019concepts}. However, the lack of standardized definitions for fractional differential operators~\cite{de2014review} presents challenges. The development of advanced operators and differential equations necessitates sophisticated techniques like the modified Galerkin methods~\cite{SAFDARI202122, SAFDARI202245}.

This paper introduces a novel fractional partial differential equation (FPDE) featuring a new numerical solution. It explores the accuracy and stability of the proposed method by examining a nonlinear convection-diffusion equation that incorporates both time and space fractional operators as follows:
\begin{equation} \label{eq1-1}
 \frac{{\partial^\alpha V(\xi ,t)}}{{\partial t^\alpha}}+ \frac{{\partial F(V)}}{{\partial \xi}} = \frac{{\partial }}{{\partial \xi}}{\left( {S\left( V \right)\frac{{\partial V(\xi,t)}}{{\partial \xi}}} \right)} + b\left( { - {{\left( { - \mathcal L } \right)}^{\frac{\beta }{2}}}} \right)V(\xi ,t)+ g(\xi,t),\quad(\xi,t)\in \mathbb{R} \times (0,T),
\end{equation}
\begin{equation*}
V(\xi,0) = {V_0}(\xi), \quad \xi \in \,\mathbb{R},
\end{equation*}
where the first term is defined as follows \cite{2}:    
\begin{equation*}
\frac{{{\partial ^\alpha }V(\xi,t) }}{{\partial {t^\alpha }}} = \frac{1}{{\Gamma (1 - \alpha )}}\int_0^t \frac{{\partial V(\xi ,\eta )}}{{\partial \eta }}\frac{{d\eta }}{{{{(t - \eta )}^{\alpha} }}} , \quad 0<\alpha\leq1, 
\end{equation*}
and $ F, S:\mathbb{R} \to \mathbb{R} $ are Lipschitz  functions such that $S\ge0$ is bounded, the term $\frac{{\partial }}{{\partial \xi}}{\left( {S\left( V \right)\frac{{\partial V(\xi,t)}}{{\partial \xi}}} \right)}$ is the nonlinear diffusion 
and $\frac{{\partial F(V)}}{{\partial \xi}}$ is  the nonlinear convection, ${b}\ge0$ is constant,
and the operator ${\left( { - \mathcal L } \right)^{\frac{\beta }{2}}}$  indicates the fractional Laplacian derivative, a generalized form of fractional spatial derivative, that is deﬁned by a singular integral \cite{27, 29}:
 \begin{equation}
{\left( \mathcal{-L}  \right)^{\frac{\beta }{2}}}(V(\xi,t))= {c_\beta }\int_{\left| z \right| > 0} {\frac{{V(\xi + z,t) - V(\xi,t)}}{{{{\left| z \right|}^{1 + \beta }}}}} dz,\,\,\,\,\,\,\ \beta \in (0,2)\,,\,\,{c_\beta} > 0.
\end{equation}

Through this article, the initial value of function $F$ is assumed zero, $ F(0)=0 $. Equation \eqref{eq1-1} with non-integer order of the operators has potential applications in various fields of study, for instance, explosives and semiconductors devices \cite{ci29}, option pricing models for mathematical ﬁnance \cite{ci8},  hydrodynamics, dislocation dynamics, molecular biology \cite{ci13}, and many other areas of research \cite{ci2,ci3,ci11}.

There are many numerical solutions of FPDEs, for example, finite difference  \cite{11}, boundary element \cite{kh2020}, and finite element methods \cite{1}. However, a few numerical methods have been developed for models with fractional Laplacian operators. A class of finite element methods \cite{304} has paved the way for developing different types of Galerkin methods, such as discontinuous Galerkin (DG) scheme for less smooth problems \cite{BASSI1997267}. The DG method has been applied for solving fractional convection-diffusion equations in \cite{8,cof,ABo}. 

Local discontinuous Galerkin (LDG) methods \cite{cof, Q, 2} have been appropriately utilized for time-dependent partial equations with higher derivatives. The main idea behind the LDG methods is converting the original equation into a first-order system by introducing some auxiliary variables for applying the DG method. Recently, this method has been exploited for a distributed-order time and space-fractional convection–diffusion with Schrödinger-type equations \cite{aboelenen2018local}. The accuracy of the LDG method significantly depends on the selection of appropriate basis functions. This paper uses Legendre basis functions to approximate Equation~\eqref{eq1-1}. The Legendre polynomials are well-known as a system of complete and orthogonal polynomials, and their mathematical properties and applications have been discussed in many contexts, such as Physics and Mathematics.

This article is compiled as follows:
 In Section \ref{Sec:1}, we give some required basic definitions. In Section \ref{Sec:2}, we use the LDG method to approximate the problem. In sections \ref{Sec:3} and \ref{Sec:4}, we prove the stability and convergence of the method.
 In Section \ref{Sec:5}, with a few numerical examples, we numerically confirm the consequence of Section \ref{Sec:4}.

\section{Preliminary definitions}\label{Sec:1}
This section introduces some basic definitions of fractional calculus \cite{4, 100}.
 Left and right Riemann-Liouville fractional integral of order $ \beta $  are defined as
\begin{equation}
\mathop {_c\mathscr L}\nolimits _\xi^{ \beta} z(\xi) = \frac{1}{{\Gamma (\beta)}}{\int_c^\xi {\left( {\xi  -\varepsilon  } \right)} ^{\beta- 1}}z(\varepsilon )d\varepsilon, \quad \xi>c, \quad \beta \in \mathbb{R}^+,
\end{equation}
\begin{equation}
\mathop {_\xi \mathscr L}\nolimits _c^{ \beta} z(\xi) = \frac{1}{{\Gamma (\beta)}}{\int_\xi^c {\left( {\varepsilon - \xi} \right)} ^{\beta- 1}}z(\varepsilon  )d\varepsilon, \quad \xi<c, \quad \beta \in \mathbb{R}^+,
\end{equation}
where $ c \in \mathbb{R} $. 
For $ \beta  \in [\gamma   - 1,\gamma ) $, the left-sided and right-sided fractional derivatives of order $ \beta $ are defined as follow:
\begin{equation*}
{}_{-\infty}D_\xi^\beta  z(\xi) = \frac{1}{{\Gamma (n - \beta )}} \frac{d^\gamma}{d\xi^\gamma}\int_{-\infty}^\xi {{{\left( {\xi - \varepsilon } \right)}^{\gamma - \beta  - 1}}} z(\varepsilon )d\varepsilon,
\end{equation*}
\begin{equation*}
\hspace*{1.2cm}{}_\xi D_\infty^\beta  z(\xi) = \frac{1}{{\Gamma (\gamma - \beta )}} \left(-\frac{d}{d\xi}\right)^\gamma \int_\xi^\infty {{{\left( {\varepsilon -\xi } \right)}^{\gamma - \beta  - 1}}} z(\varepsilon )d\varepsilon.
\vspace{+5pt}\end{equation*}
\begin{defn}
For $ 0<\beta<1 $ we define 
\begin{equation*}{\mathcal{L} _{ - \frac{\beta }{2}}}V(\xi) =  - \frac{{\mathop {{}_{ - \infty }\mathscr{T}}\nolimits_\xi^{  \beta } V(\xi) + \mathop {{}_\xi \mathscr{T}}\nolimits_\infty ^{  \beta } V(\xi)}}{{2\cos (\beta \pi /2)}}.
\end{equation*}

For $ 1<\beta<2 $, we have 
\begin{equation} \label{1-2}
 - {\left( { - \mathcal{L} } \right)^{\frac{\beta }{2}}}V(\xi) = \frac{{{d^2}}}{{d{\xi^2}}}\left( {{\mathcal{L} _{\frac{{\beta  - 2}}{2}}}V} \right) = {\mathcal{L} _{  \frac{{\beta  - 2}}{2}}}\left( {\frac{{{d^2 V}}}{{d{\xi^2}}}} \right) = \frac{d}{{d\xi}}\left( {{\mathcal{L} _{\frac{{\beta  - 2}}{2}}}\frac{dV}{{d\xi}}} \right).
\end{equation}
\end{defn}
\begin{lem}\cite{Q} \label{lem20}
The fractional integration operator $ {\mathcal{L} _{ - \beta }} $ is bounded in $ L^2(\Omega ): $
\begin{equation*}
{\left\| {{\mathcal{L} _{ - \beta }}V(\xi ,t)} \right\|_{{L^2}(\Omega )}} \le C{\left\| {V(\xi ,t)} \right\|_{{L^2}(\Omega )}}.
\end{equation*}

where $ C $ is a constant.
\end{lem}
\begin{defn}
The following common differential equation is called the Legendre differential equation:
\begin{equation}
\frac{d}{{d\xi}}\left[ {(1 - {\xi^2})\frac{d}{{d\xi}}{P_n}(\xi)} \right] + n(n + 1){P_n}(\xi) = 0.
\end{equation}
\end{defn}
The first few Legendre polynomials solutions are:
\begin{equation*}
\begin{array}{l}
\begin{array}{*{20}{c}}
n:&{{P_n}(\xi)}\\
\hline
0:&1\\
1:&\xi\\
2:&{{\textstyle{\frac{1}{2}}}\left( {3{\xi^2} - 1} \right)}\\
3:&{{\textstyle{\frac{1}{2}}}\left( {5{\xi^3} - 3\xi} \right)}\\
4:&{{\textstyle{\frac{1}{8}}}\left( {35{\xi^4} - 30{\xi^2} + 3} \right)}\\
5:&{{\textstyle{\frac{1}{8}}}\left( {63{\xi^5} - 70{\xi^3} + 15\xi} \right)}\\
\end{array}\\
\end{array}
\end{equation*}

Let us discretize the time and place of the fractional equation. We first discretize the integral interval $ [0, 1] $ by the grid $ 0 = {\pi _0} < {\pi _0} < ... < {\pi _M} = 1 $ and take 
\begin{equation}
\Delta {\pi _j} = {\pi _j} - {\pi _{j - 1}} = \frac{1}{M} = p ,\,{\alpha _j} = \frac{{{\pi _j} - {\pi _{j - 1}}}}{2} = \frac{{2j - 1}}{{2M}},\,j = 1,2,...,M,\,\,M \in \mathbb{N}. 
\end{equation}
Thus, we can write
\begin{equation}   \label{te}
\frac{{{\partial ^\alpha }V(\xi ,t)}}{{\partial {t^\alpha }}} = \sum\limits_{j = 1}^M {{W({\alpha_j })\,\mathop {{}_0^CD}\nolimits_t^{{\alpha_j }} V(\xi ,t)} \Delta {\pi _j}}   + {\rm O}({p ^2}),
\end{equation}
where \( p \) is the step size of the discretization of the numerical integration and \( W(\alpha) \) is the basis function \( {}_0^C D_t^{\alpha} V(\xi, t) \) which is the Caputo fractional derivative of order \( \alpha \)
respect to $ t $. Let $ \Delta t=\frac{T}{M} $ is the size of the grid mesh,
$ M $ an integer is positive,  $ {t_j} = j\Delta t,\,\,j = 0,1,2,...,M $  are mesh points.
\begin{lem}(see \cite {SUN2006193}) \label{lem1-1}
Assume  $ \left( {0 < \alpha  < 1} \right), y(t) \in {C^2}[0,{t_n}]. $  So that
\begin{equation*}
{\frac{1}{{\Gamma (1 - \alpha )}}\int_0^{{t_n}} {\frac{{y'(\eta )d\eta }}{{{{({t_n} - \eta )}^\alpha }}}}  - \frac{1}{\lambda }\left[ {{a_0}y({t_n}) - \sum\limits_{l = 1}^{n - 1} {({a_{n - l - 1}} - {a_{n - l}})} y({t_l}) - {a_{n - 1}}y(0)} \right]}
\end{equation*}
\begin{equation} 
 \le \frac{1}{{\Gamma (2 - \alpha )}}\left[ {\frac{{1 - \alpha }}{{12}} + \frac{{{2^{2 - \alpha }}}}{{2 - \alpha }} - (1 + {2^{ - \alpha }})} \right]\mathop {Max}\limits_{0 \le t \le {t_n}} \left| {y''(t)} \right|{(\Delta t)^{2 - \alpha }}.
\end{equation} 
\end{lem}
 For convenience, we write the formula as follows:
 \begin{equation} 
 \mathop {{}_0^CD}\nolimits_{{t_n}}^\alpha  y \approx \delta _t^\alpha {y_n} = \frac{1}{\lambda }\left( {{y_n} - \sum\limits_{l = 1}^{n - 1} {({a_{n - l - 1}} - {a_{n - l}})} {y_l} - {a_{n - 1}}{y_0}} \right). \label{ta}
 \end{equation}
 From \eqref{ta}, \eqref{te} we obtain
 \begin{equation} \label{n1}
\hspace*{-2cm} \frac{{{\partial ^\alpha }V(\xi ,t)}}{{\partial {t^\alpha }}} = \sum\limits_{j = 1}^M {W({\alpha _j})\,\mathop {{}_0^CD}\nolimits_t^{{\alpha _j}} V(\xi ,t)} \Delta {\pi _j}  \approx  \sum\limits_{j = 1}^M {  \Delta {\pi _j}W({\alpha _j})\delta _t^{{\alpha _j}}{V_n}(\xi ,t)} {\rm{ }}  
 \end{equation}
 \begin{equation} \label{n2}
 = \sum\limits_{j = 1}^M {\frac{{W({\alpha _j})\Delta {\pi_j}}}{{{\lambda _j}}}\left( {{V_n} - \sum\limits_{l = 1}^{n - 1} {\left( a^{\alpha _j}_{n - l - 1} - a^{\alpha _j}_{n - l} \right){V_l} - a_{n - 1}^{{\alpha _j}}{V_0}} } \right)},
 \end{equation}
 where $ {\lambda _j} = {\left( {\Delta t} \right)^{{\alpha _j}}}\Gamma (2 - {\alpha _j}) $  and $ a_l^{{\alpha _j}} = {(l + 1)^{1 - {\alpha _j}}} - {l^{1 - {\alpha _j}}},\,\,0 \le l \le M - 1. $
\section{The LDG method}\label{Sec:2}
The LDG method converts the original equation into a lower-order derivative system to solve higher-order derivative equations.
In this section, we define three variables $ E, L, R $, and 
 defining 
 \begin{equation*}
 \displaystyle{ S\left( V^n \right)  \frac{\partial V^n}{\partial \xi} = \left(\sqrt {S(V^n)}\right) \frac{\partial \phi (V^n)}{\partial \xi},}
 \end{equation*}
  where  $\displaystyle{ \phi(V) = \int^V {\sqrt {S(V)} }\, \mathrm{d}\xi  }$, equation \eqref{eq1-1}  is rewritten as follows:
\begin{equation*}  \label{eq4-5}
\begin{array}{l}
{{\frac{{{\partial ^\alpha }V^n(\xi ,t)}}{{\partial {t^\alpha }}}} + \left( {F(V^n) - \sqrt {S\left( V^n \right)} L-\sqrt{b}E} \right)_\xi} =g(\xi,t)  , \\
L - \phi{\left( V^n \right)_\xi} = 0,\\
E={\mathcal{L} _{  \frac{\alpha-2 }{2}}}R(\xi),\\
R=\sqrt{b}\frac{\partial}{\partial \xi}V^n.
\end{array}
\end{equation*}
We seek $  (V^n(\xi,t), L(\xi,t), E(\xi,t), R(\xi,t))  $ as an approximation of $ (V^n_h(\xi,t), L_h(\xi,t), E_h(\xi,t), R_h(\xi,t)) \in \mathbb{V}_h   $\\  so that, for any $ a(\xi), b(\xi), c(\xi), d(\xi) \in \mathbb{V}^k $, we have

\begin{equation} \label{pp}
 \begin{array}{l}
\left( \sum\limits_{j = 1}^M \Delta {\pi _j}W({\alpha _j})\delta _t^{{\alpha _j}}{V^n_h} ,a(\xi) \right)_{{I_s}} + \left( \left( F(V^n_h)-\sqrt{S(V^n_h)}L_h-\sqrt b E \right)_\xi,\frac{{\partial a}}{{\partial \xi}} \right)_{{I_s}} =  \left( g(\xi,t),a(\xi) \right)_{{I_s}},\\
\left( {{L_h},b(\xi)} \right)_{{I_s}} - \left( \phi(V^n_h)_\xi,b(\xi) \right)_{{I_s}}=0,\\
\left( {E_h},c(\xi) \right)_{{I_s}} - \left( {\mathcal{L}_{\frac{\alpha-2}{2}} R_h},c(\xi) \right)_{{I_s}} = 0,\\
{\left( {{R_h},d(\xi)} \right)_{{I_s}}} - \sqrt b{\left( {\frac{{\partial {V^n_h}}}{{\partial \xi}},d(\xi)} \right)_{{I_s}}}=0,\\
{\left( {{V^n_h}(\xi,0),a(\xi)} \right)_{{I_s}}} = {\left( {{V^n_0}(\xi),a(\xi)} \right)_{{I_s}}}.
\end{array}
\end{equation} 
denote $ {\left( {V^n,a} \right)_I} = \displaystyle{\int_I {V^n(\xi) a(\xi) \mathrm{d}\xi}}  $   is defined, that is the inner product. Now suppose:
\begin{equation*}
{V^ \pm }({\varepsilon  _s}) = \mathop {\lim }\limits_{\varepsilon \to \varepsilon_s^ \pm } V(\xi),\,\,\,\,\left\{ {\left\{ V \right\}} \right\} = \frac{{{V^ + } + {V^ - }}}{2},\,\,\,\left[\kern-0.15em\left[ V 
 \right]\kern-0.15em\right] = {V^ + } - {V^ - },
\end{equation*}
We define numerical fluxes as follows:
\begin{equation*}
\hat V = {A_V}({V^ - },{V^ + }),\,\,\,\,\,\,\,\,\,{\hat F_V} = \hat F(V_h^ - ,V_h^ + ),\,\,\,\,\,\,\,\,\,\hat L = {A_L}({L^ - },{L^ + }).
\end{equation*}
For higher derivatives, we define:
\begin{equation*}
{\hat V_{l + \frac{1}{2}}} = V_{l + \frac{1}{2}}^ - ,\,\,\,\,\,\,\,{\hat L_{l + \frac{1}{2}}} = L_{l + \frac{1}{2}}^ + \,,\,\,\,\,\,\,l = 0,1,2,...,N - 1,
\end{equation*}
and
\begin{equation*}
{\hat V_{l + \frac{1}{2}}} = V_{l + \frac{1}{2}}^ + ,\,\,\,\,\,\,\,{\hat L_{l + \frac{1}{2}}} = L_{l + \frac{1}{2}}^ - \,,\,\,\,\,\,\,l = 0,1,2,...,N - 1,
\end{equation*}
By integrating by part to \eqref{pp} and the introduced numerical fluxes, we replaced the fluxes at the interfaces, which will be obtained
\begin{equation*}
\hspace*{-1cm}\left( \sum\limits_{j = 1}^M \Delta {\pi _j}W({\alpha _j})\delta _t^{{\alpha _j}}{V^n_h}, a \right)_{{I_s}} + \left( F(V^n_h)a-\sqrt{S(V^n_h)}L_h-\sqrt b E_h, a_\xi \right)_{{I_s}} + \hat{F}(V^n_h)a\Big|_{\xi_s^+}^{\xi_{s + 1}^-}
\end{equation*}
\begin{equation}
\hspace*{-0.5cm}-\sqrt{\hat{S}(V^n_h)}\hat{L}_ha\Big|_{\xi_s^+}^{\xi_{s + 1}^-} -\sqrt b \hat{E}_h a\Big|_{\xi_s^+}^{\xi_{s + 1}^-} -\sqrt b \left( \frac{\partial E_h}{\partial \xi} ,a \right)_{{I_s}} - \left(g(\xi,t),a \right)_{{I_s}}=0, \label{s1}
\end{equation}
\vspace*{0.1cm}\begin{equation}
\hspace*{3cm}\left( {L_h,b(\xi)} \right)_{{I_s}} - {\left( {\phi(V^n_h), b_\xi} \right)_{{I_s}}} + {\hat{\phi}(V^n_h)}b|_{\xi_s^ + }^{\xi_{s + 1}^ - } = 0,\label{s2}
\vspace*{0.4cm}\end{equation}
\begin{equation}
\hspace*{4cm}\left( E_h,c(\xi) \right)_{{I_s}} - \left( \mathcal{L}_{\frac{\alpha-2}{2}} R_h,c(\xi) \right)_{{I_s}} = 0,\label{s3}
\end{equation}
\begin{equation}
\hspace*{3cm}\left( {{R_h},d(\xi)} \right)_{{I_s}} - \sqrt b {\rm{ }}{\hat V^n_h}d|_{\xi_s^ - }^{\xi_{s + 1}^ + } + \sqrt b {\left( {V^n_h,{d_\xi}} \right)_{{I_s}}} = 0,\label{s4}
\vspace*{0.5cm}\end{equation}
\begin{equation}
\hspace*{5cm}\left( {{V^n_h}(\xi,0),a(\xi)} \right) - \left( {{V^n_0},a(\xi)} \right) = 0.\label{s5}
\end{equation}
The purpose is finding $\displaystyle{ \tilde{\mathbf{A}} = {\left( {\tilde V,\tilde L,\tilde E, \tilde R} \right)^T} }$ by exploiting the LDG method such that
\begin{equation*}
 \tilde V(\xi,t) = \sum\limits_{s =1}^N {\sum\limits_{p = 1}^k {{Q_{p,s}}} (t){\zeta  _{p,s}}(\xi)} \,,\,\,\,\tilde L(\xi,t) = \sum\limits_{s =1}^N {\sum\limits_{p = 1}^k {{U_{p,s}}} (t){\zeta _{p,s}}(\xi)}, 
\end{equation*}
\begin{equation*}
 \tilde E(\xi,t) = \sum\limits_{s =1}^N {\sum\limits_{p = 1}^k {{D_{p,s}}} (t){\zeta _{p,s}}(\xi)} \,,\,\,\,\tilde R(\xi,t) = \sum\limits_{s =1}^N {\sum\limits_{p = 1}^k {{K_{p,s}}} (t){\zeta _{p,s}}(\xi)}, 
\end{equation*}
where they are functions satisfying \eqref{s1}-\eqref{s4}
for all $ a,b, c, d \in {\mathcal{P}^k}({I_s}) $, $ s \in \{1,2,...,N\} $
and we have the initial conditions for $ V, L, R $ and $ E $ from \eqref{s5}.

\section{Stability }\label{Sec:3}
This section shows that the solution of nonlinear equation \eqref{eq1-1} by the LDG method is stable.
We define:
\begin{equation*}
\hspace*{-1.5cm}
\mathscr{B}  (V^n, L, E, R; a, b, c, d) = 
\int_0^T \left( \sum\limits_{s = 1}^N \left( \sum\limits_{j = 1}^M \Delta {\pi _j}W({\alpha _j})\delta _t^{{\alpha _j}}{V^n_h},a \right)_{I_s} \right) \mathrm{d}t 
+ \int_0^T \left( \sum\limits_{s = 1}^{N-1} \left( \hat F a - \sqrt {\hat S} \hat L a - \sqrt b \hat{E} a \right) \Big|_{\xi_s^+}^{\xi_{s + 1}^-} \right) \mathrm{d}t 
\end{equation*}
\begin{equation*}
\hspace*{2.3cm}-\int_0^T{\sum\limits_{s = 1}^{N} \left(g,a\right)} + \sqrt b \int_0^T {\sum\limits_{s = 1}^N {{{(E,{a_\xi})}_{{I_s}}}} } \mathrm{d}t - \int_0^T {{{\sum\limits_{s = 1}^N {\left( {F(V^n) - \sqrt {S(V^n)} L,\frac{\partial a}{\partial \xi} } \right)}_{I_s} }}} \mathrm{d}t
\end{equation*}
\begin{equation*}
\hspace*{1.3cm}+ \int_0^T {{{\sum\limits_{s = 1}^N {\left( {L,b} \right)}_{I_s} }}} {\rm{d}}t - \int_0^T {{{\sum\limits_{s = 1}^N {\left( {\phi(V^n),{\frac{\partial b}{\partial \xi} }} \right)}_{I_s} }}} {\rm{d}}t + \int_0^T {\sum\limits_{s = 1}^{N-1} {\left( {\hat \phi(V^n)b} \right)} } |_{\xi_s^ + }^{\xi_{s + 1}^ - }{\rm{d}}t
\end{equation*}
\begin{equation} \label{sa}
\,\,\,+ \int_0^T {{{\sum\limits_{s = 1}^N {\left( {E,c} \right)}_{I_s} }}} {\rm{d}}t - \int_0^T {{{\sum\limits_{s = 1}^N {\left( {\mathcal{L} _{  \frac{\alpha-2 }{2}}}R,c \right)}_{I_s} }}}\mathrm{d}t  + \int_0^T {{{\sum\limits_{s = 1}^N {\left( {R, d} \right)}_{I_s} }}} {\rm{d}}t
\end{equation}
\begin{equation*} 
\hspace*{-2.2cm}+ \int_0^T {{{\sum\limits_{s = 1}^N {\sqrt b \left( {V^n,{\frac{\partial d}{\partial \xi}}} \right)}_{I_s} }}} {\rm{d}}t + \int_0^T {\sum\limits_{s = 1}^{N-1} {\sqrt b \hat V^n \, d} } |_{\xi_s^ + }^{\xi_{s + 1}^ - }{\rm{d}}t.
\end{equation*}

Notice \( \mathbb{B} (V^n, L, E, R, a, b, c, d) =0 \) for any $ \left( {a, b, c, d} \right) $ if $ \left( {V^n, L, E, R} \right) $ is a solution. By considering the fluxes,
\begin{equation*}
 {{\hat V^n}_{s + 1}} = (V^n)_{s + 1}^ -,\quad \hat{L}_{s+1}=L^+_{s+1}, \quad
\hat{E}_{s+1}=E^+_{s+1}, \quad \hat{\phi}(V^n)_{s+1} =\phi((V^n)^+_{s+1}), \quad 1\le s \le N - 1,
\end{equation*}
in boundary conditions, we define the following flux:
\begin{equation*}
{\hat V^n_{N + 1}} = V^n(b,t),\quad \quad {\hat E_{N + 1}} = E_{N + 1}^ -  + \frac{\beta }{h}[{V^n_{N + 1}}].
\end{equation*} 
 So we can write:
\begin{equation*} 
\hspace{-3cm}\mathscr{B}(V^n, L, E, R; a, b, c, d) = \int_0^T {\sum\limits_{s = 1}^N \left( \sum\limits_{j = 1}^M {  \Delta {\pi _j}W({\alpha _j})\delta _t^{{\alpha _j}}{V^n}} {\rm{ }},a \right)_{{I_s}}} {\rm{d}}t - \int_0^T {\sum\limits_{s = 1}^N {\left( {F(V^n),{\frac{\partial a}{\partial \xi}}} \right)}_{I_s} {\rm{d}}t }
\end{equation*}
\begin{equation*}
+ \sqrt b \int_0^T {\sum\limits_{s = 1}^N {{{\left(E,{\frac{\partial a}{\partial \xi}}\right)_{I_s}}}} }\mathrm{d}t + {\int_0^T {\sum\limits_{s = 1}^N {\left( {\sqrt {S(V^n)} L,{\frac{\partial a}{\partial \xi}}} \right)}_{I_s} }}\mathrm{d}t 
\end{equation*}
\begin{equation*}
\hspace*{1.2cm}+ {\int_0^T {\sum\limits_{s = 1}^N {\left( {L,b} \right)}_{I_s} } }\mathrm{d}t + {\int_0^T {\sum\limits_{s = 1}^N {\left( g, a \right)}_{I_s} } }\mathrm{d}t - {\int_0^T {\sum\limits_{s = 1}^N {\left( {\phi(V^n),{\frac{\partial b}{\partial \xi}}} \right)_{I_s}} } }\mathrm{d}t
\end{equation*}
\begin{equation*}
\hspace*{2cm}  + {\int_0^T {\sum\limits_{s = 1}^N {\left( {E ,c} \right)_{I_s}} }}\mathrm{d}t - \int_0^T {\sum\limits_{s = 1}^N {\left( {\mathcal{L} _{\frac{\alpha-2}{2}}R,c} \right)_{I_s}} }\mathrm{d}t + {\int_0^T {\sum\limits_{s = 1}^N {\left( {R ,d} \right)_{{I_s}}} } }\mathrm{d}t
\end{equation*}
\begin{equation*}
\hspace*{-0.5cm} + {\int_0^T {\sum\limits_{s = 1}^N {\sqrt b \left( {V^n,{\frac{\partial d}{\partial \xi}}} \right)_{{I_s}}} } }\mathrm{d}t - \int_0^T {\sum\limits_{s = 1}^{N - 1} {{{\hat F}_{_{s + 1}}}{{\left[\kern-0.15em\left[ {a} 
 \right]\kern-0.15em\right]}_{s + 1}}} } \mathrm{d}t
\end{equation*}
\begin{equation*}
\hspace*{0.2cm}+\int_0^T {\sum\limits_{s = 1}^{N - 1} {{{(\sqrt {\hat S} \hat L)}_{s + 1}}{{\left[\kern-0.15em\left[ {a} 
 \right]\kern-0.15em\right]}_{s + 1}}} } \mathrm{d}t - \int_0^T {\sum\limits_{s = 1}^{N - 1} {{{\hat \phi}_{s + 1}}{{\left[\kern-0.15em\left[ {b} 
 \right]\kern-0.15em\right]}_{r + 1}}} } \mathrm{d}t
\end{equation*}
\begin{equation*}
\hspace*{0.3cm} + \sqrt b \int_0^T {\sum\limits_{s = 1}^{N - 1} {{{\hat E}_{s + 1}}{{\left[\kern-0.15em\left[ {a} 
 \right]\kern-0.15em\right]}_{s + 1}}} } \mathrm{d}t - \int_0^T {\sum\limits_{s = 1}^{N - 1} {\sqrt b {{\hat V^n}_{s + 1}}{{\left[\kern-0.15em\left[ {d} 
 \right]\kern-0.15em\right]}_{r + 1}}} } \mathrm{d}t
\end{equation*}
\begin{equation*}
\hspace*{3.45cm}- \int_0^T {\left( {{{\hat F}_1}a_1^ +  - {{\hat F}_{N+1}}a_{N+1}^ - } \right)} \mathrm{d}t + \int_0^T {\left( {\sqrt {{{\hat S}_1}} {{\hat L}_1}a_1^ +  - \sqrt {{{\hat S}_{N+1}}} {{\hat L}_{N+1}}a_{N+1}^ - } \right)} \mathrm{d}t
\end{equation*}
\begin{equation*}
\hspace*{2.45cm} - \int_0^T {\left( {{{\hat \phi}_1}b_1^ +  - {{\hat \phi}_{N+1}}b_{N+1}^ - } \right)} \mathrm{d}t + \int_0^T {\left( {\sqrt b {{\hat E}_1}a_1^ +  - \sqrt b {{\hat E}_{N+1}}a_{N+1}^ - } \right)} \mathrm{d}t
\end{equation*}
\begin{equation}  \label{zeyad}
\hspace*{-2.7cm} - \int_0^T {\left( {\sqrt b {{\hat V^n}_1}d_1^ +  - {{\hat V^n}_{N+1}}d_{N+1}^ - } \right)\mathrm{d}t}.
\end{equation}
\begin{lem} \label{Lemma3}
By setting $ (a, b, c, d)=(V^n , L,-R, E) $ in \eqref{zeyad} and defining $\displaystyle{ \psi (V^n)=\int^{V^n} F(V^n)\, \rm{d}V^n} $, we achieve the following result
\begin{equation*}
\hspace{-3.4cm}\mathscr{B}(V^n, L, E, R; V^n, L, - R, E) = \int_0^T {\sum\limits_{s = 1}^N \left( \sum\limits_{j = 1}^M {  \Delta {\pi_j}W({\alpha _j})\delta _t^{{\alpha _j}}{V^n}} {\rm{ }},V^n \right)_{{I_s}} {\rm{d}}t} + \int_0^T {\sum\limits_{s = 1}^N {(L,L)}_{I_s} {\rm{d}}t}
\end{equation*}
\begin{equation*}
\hspace*{2.8cm} + \int_0^T{\sum\limits_{s = 1}^N \left( \mathcal{L}{\frac{\alpha-2}{2}}(R, R) \right){I_s} \mathrm{d}t} + \int_0^T {\frac{\sqrt b}{h} \beta \left( V^{n^-}_{N + 1} \right)^2 \mathrm{d}t}
\end{equation*}
\begin{equation*}
\hspace*{3cm}+ \int_0^T {\left( {\psi {{(V^n)}_1} - \psi{{(V^n)}_{N + 1}} - {{(\hat FV^n)}_1} + {{(\hat FV^n)}_{N + 1}}} \right)\mathrm{d}t}
\end{equation*}
\begin{equation*}
\hspace*{0.75cm} + \int_0^T {\sum\limits_{s = 1}^{N-1} {\left( {{{\left[\kern-0.15em\left[ {\psi(V^n)} 
 \right]\kern-0.15em\right]}_{s + 1}} - \hat F{{\left[\kern-0.15em\left[ {V^n} 
 \right]\kern-0.15em\right]}_{s + 1}}} \right)} } \mathrm{d}t.
\end{equation*}
\end{lem}
\begin{proof}
If we suppose $ (a, b, c, d)=(V^n, L,-R, E) $ in \eqref{zeyad}, and apply the integration by parts formula
\begin{equation*}
{\left(\phi(V),\frac{\partial L}{\partial \xi}\right)}_{I_s}+{\left( \frac{\partial \phi(V^n)}{\partial \xi},L\right)}_{I_s}=\phi(V^n)L|_{\xi_s^+}^{\xi_{s+1}^-},
\end{equation*}
\begin{equation*}
\hspace*{0.75cm}{{{\left(E,{\frac{\partial V^n}{\partial \xi}}\right)}_{{I_s}}}}+{{{\left(\frac{\partial E}{\partial \xi},{V^n}\right)}_{{I_s}}}}=(E V^n)|_{\xi_s^+}^{\xi_{s+1}^-},
\end{equation*}
the interface condition can be obtained 
\begin{align*}
&\sum\limits_{s = 1}^N {{{\left( {\sqrt {S(V^n)} L,{\frac{\partial V^n}{\partial \xi}}} \right)}_{{I_s}}}} + \sum\limits_{s = 1}^N {{{\left( {\phi(V^n),{\frac{\partial L}{\partial \xi}}} \right)}_{{I_s}}}}  \\
&+ \sum\limits_{s = 1}^{N - 1} {\hat \phi(V^n){{\left[\kern-0.15em\left[ {L} \right]\kern-0.15em\right]}_{{{s + 1}}}}} + \sum\limits_{s = 1}^{N - 1} {{{(\sqrt {\hat S} \hat L)}_{s + 1}}{{\left[\kern-0.15em\left[ {a} \right]\kern-0.15em\right]}_{s + 1}}} \\
&= \phi(V^{n+}_1)L_1^+ - \phi(V^{n^-}_{N+1})L_{N+1}^-,
\end{align*}
\begin{align*}
& {\sum\limits_{s = 1}^N {{{\sqrt{b}\left(E,{\frac{\partial V^n}{\partial \xi}}\right)}_{{I_s}}}} } + {\sum\limits_{s = 1}^N {{{\sqrt{b}\left(\frac{\partial E}{\partial \xi},{V^n}\right)}_{{I_s}}}} } \\
&+  {\sum\limits_{s = 1}^{N-1} \sqrt{b}{ E_{s+1}^+{{\left[\kern-0.15em\left[ {V^n} 
 \right]\kern-0.15em\right]}_{s + 1}}} + {\sum\limits_{s = 1}^{N-1} \sqrt{b} V^{n+}_{s+1}{{\left[\kern-0.15em\left[ {E}
 \right]\kern-0.15em\right]}_{s + 1}}} } \\
&=\sqrt b (V^n_1)^+E_1^+-\sqrt b (V^n_{N+1})^-E_{N+1}^- .
\end{align*}
Then we have
\begin{equation*}
\hspace{-.7cm}\mathscr{B}(V^n, L, E, R;V^n, L, -R, E) = \int_0^T \sum\limits_{s = 1}^N \left( \sum\limits_{j = 1}^M \Delta \pi _j W(\alpha j) \delta t^{\alpha j} V^n, V^n \right){I_s} \mathrm{d}t - \int_0^T \sum\limits{s = 1}^N \left( F(V^n),\frac{\partial V^n}{\partial \xi} \right){I_s} \mathrm{d}t
\end{equation*}
\begin{equation*}
\hspace*{5cm}
+ \int_0^T \sum\limits_{s = 1}^N \left( L, L \right)_{I_s} \mathrm{d}t  
+ \int_0^T \sum\limits_{s = 1}^N \left( \mathcal{L}_{\frac{\alpha-2}{2}} R, R \right)_{I_s} \mathrm{d}t 
- \int_0^T \sum\limits_{s = 1}^{N - 1} \hat F_{s + 1} \left[\kern-0.15em\left[ V^n \right]\kern-0.15em\right]_{s + 1} \mathrm{d}t  
\end{equation*}
\begin{equation} \label{f1}
\quad \quad \quad \quad \,\,\,\, \quad+ \int_0^T \frac{\sqrt b}{h} \varepsilon \left( V^{n-}_{N + 1} \right)^2 \mathrm{d}t - \int_0^T \left( \hat F_1 \hat V^n_1 - \hat F_{N + 1} \hat V^n_{N + 1} \right) \mathrm{d}t.
\end{equation}
Define $\displaystyle { \psi(V^n)=\int ^{V^n} F(V^n)\, \rm{d}V^n }$, then
\begin{equation} \label{f2}
{\sum\limits_{s = 1}^N {\left( {F(V^n),{\frac{\partial V^n}{\partial \xi} }} \right)} _{{I_s}}} = \sum\limits_{s = 1}^N {\psi (\xi)|_{\xi_s^ + }^{\xi_{s + 1}^ - }}  =  - \sum\limits_{s = 1}^{N-1} {{{\,\psi \left[\kern-0.15em\left[  (V^n) 
 \right]\kern-0.15em\right]}_{s + 1}}}  - \psi{(V)^n_1} + \psi{(V)^n_{N + 1}}.
\end{equation}
Finally, using equations \eqref{f1} and \eqref{f2} proves the Lemma.
\end{proof}
\begin{thm}
The semi-discrete scheme \eqref{s1}-\eqref{s5} is stable, and $\forall T>0 $ we have $ \left\| {{V^n_h}(\xi,T)} \right\| \le \left\| {{V^n_0}(\xi)} \right\| $.
\end{thm}
\begin{proof}
Using the uniformity property of the flux function $ \hat F\left({(V^n)^ - },{(V^n)^ + }\right) $ 
we have 
\begin{equation*}
\psi \left[\kern-0.15em\left[ {\xi} 
 \right]\kern-0.15em\right]{_{s + 1}} - F{\left[\kern-0.15em\left[ {\xi} 
 \right]\kern-0.15em\right]_{s + 1}} > 0,\,\,1 \le s \le N - 1
\end{equation*}
Using  Galerkin orthogonality,\\ $ \mathscr{B}(V^n_h, L_h, E_h, R_h,V^n_h, L_h, - R_h, E_h) =0$, Lemma \ref{Lemma3} yields
\begin{equation*}
\int_0^T \sum\limits_{s = 1}^N \left( \sum\limits_{j = 1}^M \Delta {\pi _j}W({\alpha _j})\delta _t^{{\alpha _j}}{V^n}, V^n \right)_{I_s} \mathrm{d}t  + \int_0^T \sum\limits_{s = 1}^N (L,L)_{I_s} \mathrm{d}t + \int_0^T \sum\limits_{s = 1}^N \left( \mathcal{L}_{\frac{\alpha-2}{2}} R, R \right)_{I_s} \mathrm{d}t 
\end{equation*}
\begin{equation*}
 + \int_0^T {\frac{{\sqrt b }}{h}} \varepsilon \left( {V_{N + 1}^ - } \right)^2\mathrm{d}t+ \int_0^T {\left( {\psi{{(V)}_1} - \psi{{(V)}_{N + 1}} - {{(\hat F V)}_1} + {{(\hat F V)}_{N + 1}}} \right)\mathrm{d}t}  \le 0.
\end{equation*}
On the other hand,  according to equations \eqref{n1}, \eqref{n2} we have
\begin{equation*}
\left( \sum\limits_{j = 1}^M \frac{{W({\alpha _j})\Delta {\pi _j}}}{{\lambda _j}} V^n, V_h^n \right) \le \left( \sum\limits_{j = 1}^M \frac{{W({\alpha _j})\Delta {\pi _j}}}{{\lambda _j}} \sum\limits_{l = 1}^{n - 1} \left( a_{n - l - 1}^{{\alpha _j}} - a_{n - l}^{{\alpha _j}} \right) V_l, V_h^n \right)
\end{equation*}
\begin{equation*}
 \hspace*{3cm} + \left( \sum\limits_{j=1}^M \frac{{W({\alpha _j})\Delta {\pi _j}}}{{\lambda _j}} a_{n - 1}^{{\alpha _j}} V_0^n, V_h^n \right)
\end{equation*}
By considering Cauchy-Schwarz inequality, we have
\begin{equation*}
\left\| {V_h^n} \right\|_{{L^2}(\Omega )}^2 \le {c_1} {\sum\limits_{j = 1}^M {\frac{{W({\alpha _j})\Delta {\pi _j}}}{{{\lambda _j}}}Q} \left( {\sum\limits_{l = 1}^{n - 1} {\left( {a_{n - l - 1}^{{\alpha _j}} - a_{n - l}^{{\alpha _j}}} \right)} } \right)} \left\| {V_h^l} \right\|_{{L^2}(\Omega )}^2\left\| {V_h^n} \right\|_{{L^2}(\Omega )}^2
\end{equation*}
\begin{equation}
 + {c_2} {\sum\limits_{j = 1}^M {\frac{{W({\alpha _j})\Delta {\pi _j}}}{{{\lambda _j}}}Q} } a_{n - 1}^{{\alpha _j}}\left\| {V_h^0} \right\|_{{L^2}(\Omega )}^2\left\| {V_h^n} \right\|_{{L^2}(\Omega )}^2
\end{equation}
where 
\begin{equation*}
Q = \left(\sum\limits_{j = 1}^M {\frac{{W({\alpha _j})\Delta {\pi _j}}}{{{\lambda _j}}}} \right)^{-1}
\end{equation*}
Assuming $  c$ is very small such that $ 1 - cQ > 0$, we have
\begin{equation} \label{n3}
\left\| {V_h^n} \right\|_{{L^2}(\Omega )}\le C\left(  {\sum\limits_{j = 1}^M {\frac{{W({\pi _j})\Delta {\pi _j}}}{{{\lambda _j}}}Q} \sum\limits_{l = 1}^{n - 1} {\left( {a_{n - l - 1}^{{\alpha _j}} - a_{n - l}^{{\alpha _j}}} \right)} }  \right)\left\| {V_h^l} \right\|_{{L^2}(\Omega )}^{}
\end{equation}
The theorem is proved for $ n=0$. Suppose that it is valid for $ n=1,2,3,...,m-1. $ Then, by \eqref{n3},we can write:
\begin{equation*}
\left\| {V_h^m} \right\|_{{L^2}(\Omega )} \le C\left( {\sum\limits_{j = 1}^M {\frac{{W({\alpha _j})\Delta {\pi _j}}}{{{\lambda _j}}}Q} \sum\limits_{l = 1}^{m - 1} {\left( {a_{n - l - 1}^{{\alpha _j}} - a_{n - l}^{{\alpha _j}}} \right)} } \right)\left\| {V_h^l} \right\|_{{L^2}(\Omega )}^{} + \sum\limits_{j = 1}^{m - 1} {\frac{{W({\alpha _j})\Delta {\pi _j}}}{{{\lambda _j}}}Q} a_{n - 1}^{{\alpha _j}}\left\| {V_h^0} \right\|_{{L^2}(\Omega )}^{}
\end{equation*}
\begin{equation}
\hspace*{2.3cm} \le C\left( {\sum\limits_{j = 1}^M {\frac{{W({\alpha _j})\Delta {\pi _j}}}{{{\lambda _j}}}Q} \sum\limits_{l = 1}^{m - 1} {\left( {a_{n - l - 1}^{{\alpha _j}} - a_{n - l}^{{\alpha _j}}} \right)} } \right)\left\| {V_h^0} \right\|_{{L^2}(\Omega )}^{} + \sum\limits_{j = 1}^{m - 1} {\frac{{W({\alpha _j})\Delta {\pi _j}}}{{{\lambda _j}}}Q} a_{n - 1}^{{\alpha _j}}\left\| {V_h^0} \right\|_{{L^2}(\Omega )}^{}
\end{equation}
\begin{equation*}
\hspace*{-10.80cm}=\left\| {V_h^0} \right\|_{{L^2}(\Omega )}
\end{equation*}
\end{proof}
\section{Error estimation}\label{Sec:4}
To estimate the error, we assume  $ F = 0 $, $ S  \equiv 1 $ and $ \phi(V)=V $. For fractional diffusion, \eqref{s1}-\eqref{s5} reduce to
\begin{equation}  \label{ss1}
\left( \sum\limits_{j = 1}^M {  \Delta {\pi _j}W({\alpha _j})\delta _t^{{\alpha _j}}{V^n_h}} {\rm{ }} ,a(\xi) \right)_{{I_s}}  - ( {L_h,\frac{\partial a}{\partial \xi} } )_{{I_s}}+{\left( {\kappa^n (\xi ),a(\xi )} \right)_{{I_s}}}   - \hat{L}_h a|_{\xi_s^ + }^{\xi_{s + 1}^-} +\sqrt b {( E_h, \frac{\partial a}{\partial \xi}} )_{{I_s}} -\sqrt b {( \hat{E}_h a} )|_{\xi_s^ + }^{\xi_{s + 1}^-}= 0, 
\end{equation}
\begin{equation}
\left( {L_h,b(\xi)} \right)_{{I_s}} - {\left( {V^n_h, \frac{\partial b}{\partial \xi} } \right)_{{I_s}}} + {\hat{V}^n_h}b|_{\xi_s^ + }^{\xi_{s + 1}^ - } = 0,
\end{equation}
\begin{equation}
\left( E_h,c(\xi) \right)_{I_s} - \left( \mathcal{L}_{\frac{\alpha-2}{2}} R_h,c(\xi) \right)_{I_s} = 0,
\end{equation}
\begin{equation}
{\left( {{R_h},d(\xi)} \right)_{{I_s}}} - \sqrt b {\rm{ }}{\hat V^n_h}d|_{\xi_s^ - }^{\xi_{s + 1}^ + } + \sqrt b {\left( {V^n_h,{\frac{\partial d}{\partial \xi} }} \right)_{{I_s}}} = 0, 
\end{equation}
\begin{equation}
\left( {{V^n_h}(\xi,0),a(\xi)} \right) - \left( {{V^n_0},a(\xi)} \right) = 0.  \label{ss5}
\end{equation}
As a result, the design can be written as follows:
\begin{equation*}
\mathscr{B}(V^n, L, E, R; a, b, c, d) = \int_0^T \sum\limits_{s = 1}^N \left( \sum\limits_{j = 1}^M \Delta {\pi _j}W({\alpha _j})\delta _t^{{\alpha _j}}{V^n}, a(\xi) \right)_{I_s} \mathrm{d}t - \int_0^T \sum\limits_{s = 1}^N (L, \frac{\partial a}{\partial \xi})_{I_s} \mathrm{d}t + \left( \kappa^n (\xi ), a(\xi) \right)_{I_s}
\end{equation*}
\begin{equation*}
\hspace*{4cm} + \int_0^T \sum\limits_{s = 1}^{N-1} L_{s + 1}^+ \left[\kern-0.15em\left[ a \right]\kern-0.15em\right]_{s + 1} \mathrm{d}t 
+ \sqrt{b} \int_0^T \sum\limits_{s = 1}^N \left(E, \frac{\partial a}{\partial \xi} \right)_{I_s} \mathrm{d}t 
+ \int_0^T \sum\limits_{s = 1}^{N-1} \sqrt{b} E_{s + 1}^+ \left[\kern-0.15em\left[ a \right]\kern-0.15em\right]_{\xi_s^+}^{\xi_{s+1}^-} \mathrm{d}t
\end{equation*}
\begin{equation*}
\hspace*{3.3cm}+ \int_0^T \sum\limits_{s = 1}^N \left( L, b(\xi) \right)_{I_s} \mathrm{d}t 
- \int_0^T \sum\limits_{s = 1}^N \left( V^n, \frac{\partial b}{\partial \xi} \right)_{I_s} \mathrm{d}t 
- \int_0^T \sum\limits_{s = 1}^{N-1} (\hat{V}^n_{s + 1})^+ \left[\kern-0.15em\left[ b \right]\kern-0.15em\right]_{s + 1} \mathrm{d}t
\end{equation*}
\begin{equation*}
\hspace*{3cm} +\int_0^T \sum\limits_{s = 1}^N \left( E, c(\xi) \right)_{I_s} \mathrm{d}t 
- \int_0^T \sum\limits_{s = 1}^N \left( {\mathcal{L} _{\frac{\alpha-2}{2}} R}, c \right)_{I_s} \mathrm{d}t 
+ \int_0^T \sum\limits_{s = 1}^N \left( R, d(\xi) \right)_{I_s} \mathrm{d}t
\end{equation*}
\begin{equation*} 
\hspace*{1.5cm} + \int_0^T {\sum\limits_{s = 1}^N {\sqrt b {{\left( {V^n, \frac{\partial d}{\partial \xi} } \right)}_{{I_s}}}} }+ \sqrt b \int _0^T {\sum\limits_{s=1}^{N-1} (\hat{V}^n_{s + 1})^ + {\left[\kern-0.15em\left[ {d} 
 \right]\kern-0.15em\right]_{s + 1}}|_{\xi_s^ - }^{\xi_{s + 1}^ + }}
\end{equation*}
\begin{equation}  \label{eqB1}
\hspace*{3.5cm}+\sqrt{b}\int_0^T E_1^+a_1^+ \mathrm{d}t +\frac{\sqrt{b}\beta}{h}\int_0^T( V^n_{N+1})^-a_{N+1}^- \mathrm{d}t -\sqrt{b}\int_0^T E_{N+1}^-a_{N+1}^- \mathrm{d}t,
\end{equation}
where 
\begin{equation} \label{n4}
\left| {{\kappa ^n}(\xi )} \right| = \left| {O\left( {{{(\Delta t)}^{2 - {\alpha _j}}} + {p ^2}} \right)} \right| \le c\left( {{{(\Delta t)}^{1 + \frac{p }{2}}} + {p ^2}} \right),
\end{equation}
such that
\begin{equation} \label{n5}
1 + \frac{p }{2} = 2 - Mp  + \frac{p }{2} \le 2 - {\alpha _j} = 2 - jp  + \frac{p }{2} \le 2 - p  + \frac{p }{2} = 2 - \frac{p }{2}.
\end{equation}

We define projection $ \mathscr{S}^{ \pm } $ in $ V^k $ such that
\begin{equation} \label{pro2}
\int_{{I_s}} {\left( {\mathscr{S}^{ \pm }e(x) - e(x)} \right){\zeta _{ij}}(\xi)\mathrm{d}\xi = 0}.  \quad j =1,2,.., N, \quad i=0,1,...k-1
\end{equation}
 and $\mathscr{S}^{ \pm }V^n_{s+1}=V^n(\xi^{\pm}_{s+1}) $. 
Suppose $ {e_{V^n}} = V^n - {V^n_h},\,\,\,{e_E} = E - {E_h},\,\,\,{e_L} = L - {L_h}$, and ${e_R} = R - {R_h} $, then  $ {\mathscr{S}^ - }{e_{V^n}}{\rm{ }} = {\rm{ }}{{\rm{\mathscr{S}}}^ - }V^n{\rm{ }} - {V^n_h},\,\,\,{\mathscr{S}^ + }{e_E}{\rm{ }} = {\rm{ }}{{\mathscr{S}}^ + }{E } - E_h,\,\,\,{\mathscr{S}^ + }{e_L}{\rm{ }} = {\rm{ }}{{\rm{\mathscr{S}}}^ + }L{\rm{ }} - {\rm{ }}{{L}}_h$, and $\mathscr{S}{e_R}{\rm{ }} = {\rm{ \mathscr{S}}}R{\rm{ }} - {R_h}{\rm{ }} $ for all $ (a, b, c, d) \in {H^1}(\Omega ,\mathcal{T}) \times {L^2}(\Omega ,\mathcal{T}) \times {L^2}(\Omega ,\mathcal{T}) \times {L^2}(\Omega ,\mathcal{T}) $,
\begin{equation}
\mathscr{B}(V^n, L, E, R; a, b, c, d) = \mathscr{S}(a, b, c, d).
\end{equation}
Hence,\,\, $ \mathscr{B}\left( {{e_{V^n}},{e_L},{e_E},{e_R}; a, b, c, d} \right) = 0 $ and we gain
\begin{equation*}
\hspace{-7cm}\mathscr{B}\left( {{\mathscr{S}^ - }{e_{V^n}},{\mathscr{S}^ + }{e_L},{\mathscr{S}^ + }{e_E},\mathscr{S}{e_R};{\mathscr{S}^ - }{e_{V^n}},{\mathscr{S}^ + }{e_L}, - \mathscr{S}{e_R},{\mathscr{S}^ + }{e_E}} \right)
\end{equation*}
\begin{equation*}
\hspace*{+1.1cm}= \mathscr{B}\left( {{\mathscr{S}^ - }{e_{V^n}} - {e_{V^n}},{\mathbb{S}^ + }{e_L} - {e_L},{\mathbb{S}^ + }{e_E} - {e_E},\mathscr{S}{e_R} - {e_R};{\mathscr{S}^ - }{e_{V^n}},{\mathscr{S}^ + }{e_L}, - \mathscr{S}{e_R},{\mathscr{S}^ + }{e_E}} \right)
\end{equation*}
\begin{equation*}
\hspace*{0.5cm}= \mathscr{B}\left( {{\mathscr{S}^ - }{V}^n{\rm{ }} - {V^n},{\mathscr{S}^ + }L{\rm{ }} - {\rm{ L}},{\mathscr{S}^ + }E - E,\mathscr{S}R - R;{\mathscr{S}^ - }{e_{V^n}},{\mathscr{S}^ + }{e_L}, - \mathscr{S}{e_R},{\mathscr{S}^ + }{e_E}} \right).
\end{equation*}
Substitute $ \left( {{\mathscr{S}^ - }{V^n}{\rm{ }} - {V^n},{\mathscr{S}^ + }L{\rm{ }} - {\rm{ L}},{\mathscr{S}^ + }E - E,\mathscr{S}R - R;{\mathscr{S}^ - }{e_{V^n}},{\mathscr{S}^ + }{e_L}, - \mathscr{S}{e_R},{\mathscr{S}^ + }{e_E}} \right) $ into \eqref{eqB1} we come to the following Lemma:
\begin{lem} \label{good}
Form \eqref{eqB1} can be written as follows.
\begin{equation*}
\hspace{-5cm}\mathscr{B}\left( {{\mathscr{S}^ - }{V^n}{\rm{ }} - {V^n},{\mathscr{S}^ + }L{\rm{ }} - {\rm{ L}},{\mathscr{S}^ + }E - E,\mathscr{S}R - R;{\mathscr{S}^ - }{e_{V^n}},{\mathscr{S}^ + }{e_L}, - \mathscr{S}{e_R},{\mathscr{S}^ + }{e_E}} \right)
\end{equation*}
\begin{equation*}
\le \int_0^T \sum\limits_{s = 1}^N \left( \sum\limits_{j = 1}^M \Delta {\pi _j}W({\alpha _j})\delta _t^{{\alpha _j}}\left(\mathscr{S^-} V^n-V^n\right), \mathscr{S}^ - e_{V^n} \right)_{I_s} \mathrm{d}t + C_{T,a,b} \left( h^{2k+2} + (\Delta t)^{4 + p} + p^4 \right)
\end{equation*}
\begin{equation*}
+ \frac{1}{C_{T,a,b}} \int _0^T \sum\limits_{s=1}^N \left\| {\mathscr{S}e_L} \right\|_{I_s}^2 \mathrm{d}t+\int_0^T \frac{\sqrt{b}\beta}{h}|(\mathscr{S^-}e_{V^n})_{N+1}|^2 \mathrm{d}t + \int_0^T \sum\limits_{s = 1}^N {{ {{\left\|{\mathscr{S}^ + {e_L}}\right\| }_{{I_s}}^2}\mathrm{d}t}},
\end{equation*}
where $ C_{T,a,b} $ is independent of $ h $, but may depend on $ T $ and $ \Omega $.
\end{lem}
\begin{proof}
From \eqref{eqB1} we have
\begin{equation*}
\hspace{-4cm}\mathscr{B}\left( {{\mathscr{S}^ - }V^n{\rm{ }} - V^n,{\mathscr{S}^ + }L{\rm{ }} - {\rm{ L}},{\mathscr{S}^ + }E - E,\mathscr{S}R - R;{\mathscr{S}^ - }{e_{V^n}},{\mathscr{S}^ + }{e_L}, - \mathscr{S}{e_R},{\mathscr{S}^ + }{e_E}} \right)
\end{equation*}
\begin{equation*}
= \int_0^T \sum\limits_{s = 1}^N \left( \sum\limits_{j = 1}^M \Delta {\pi _j}W({\alpha _j})\delta _t^{{\alpha _j}}\left(\mathscr{S^-} V^n-V^n\right), \mathscr{S}^ - e_{V^n} \right)_{I_s} \mathrm{d}t + \int_0^T \sum\limits_{s = 1}^N \left( \mathscr{S}^ + L - L, \frac{\partial (\mathscr{S}^ - e_{V^n})}{\partial \xi} \right)_{I_s} \mathrm{d}t
\end{equation*}
\begin{equation*}
\hspace*{1cm} + \sqrt b \int_0^T \sum\limits_{s = 1}^N \left( \mathscr{S}^+ E - E, \frac{\partial (\mathscr{S}^- e_{V^n})}{\partial \xi} \right)_{I_s} \mathrm{d}t + \int_0^T \sum\limits_{s = 1}^N \left( \mathscr{S}^+ L - L, \mathscr{S}^+ e_L \right)_{I_s} \mathrm{d}t
\end{equation*}
\begin{equation*}
 + \int_0^T {\sum\limits_{s = 1}^N {{{\left( {{\mathscr{S}^ - }{V^n} - {V^n},{{({\mathscr{S}^ + }{e_L})}_\xi}} \right)}_{{I_s}}}\mathrm{d}t}  - } \int_0^T {\sum\limits_{s = 1}^N {{{\left( {{\mathscr{S}^ + }E - E,{{\mathscr{S}{e_R}}}} \right)}_{{I_s}}}\mathrm{d}t}  } 
\end{equation*}
\begin{equation*}
\hspace*{0.8cm} + \int_0^T \sum\limits_{s = 1}^N \left( \mathcal{L}_{\frac{\alpha-2}{2}} (\mathscr{S} R - R), \mathscr{S} e_R \right)_{I_s} \mathrm{d}t + \int_0^T \sum\limits_{s = 1}^N \left( (\mathscr{S} R - R), \mathscr{S}^+ e_E \right)_{I_s} \mathrm{d}t
\end{equation*}
\begin{equation*}
\hspace*{1.3cm}-\int_0^T {\sum\limits_{s = 1}^N {{{\left( {{\mathscr{S}^ - }V^n - V^n,{{({\mathscr{S}^ + }{e_E})}_\xi}} \right)}_{{I_s}}}\mathrm{d}t} } - \int_0^T {\sum\limits_{s = 1}^{N} {{{({\mathscr{S}^ + }L - L)}^ +_{s+1} }
     { \left[\kern-0.25em\left[ {{\mathscr{S}^ - }{e_{V^n}}} 
 \right]\kern-0.25em\right]     _{s + 1}}\mathrm{d}t}  }
\end{equation*}
\begin{equation*}
\hspace*{2.4cm} -\sqrt{b} \int_0^T \sum\limits_{s = 1}^{N -1} ({\mathscr{S}^ + }E - E)_{s + 1}^+ \left[\kern-0.25em\left[ {\mathscr{S}^ - }{e_{V^n}} \right]\kern-0.25em\right]_{s + 1} \mathrm{d}t - \int_0^T \sum\limits_{s = 1}^{N} ({\mathscr{S}^ - }V^n - V^n)_{s + 1}^- \left[\kern-0.25em\left[ {\mathscr{S}^ + }{e_L} \right]\kern-0.25em\right]_{s + 1} \mathrm{d}t
\end{equation*}
\begin{equation*}
\hspace*{1.5cm} -\sqrt{b}\int_0^T {\sum\limits_{s = 1}^{N-1} {({\mathscr{S}^ - }V^n - V^n)_{s + 1}^ -  {\left[\kern-0.25em\left[ {{\mathscr{S}^ + }{e_E}} 
 \right]\kern-0.25em\right]_{s + 1}}\mathrm{d}t}  } +\sqrt{b}\int_0^T (\mathscr{S}^+E-E)_{1}^+  \left[\kern-0.25em\left[ {\mathscr{S}^-e^+_{V^n}} 
 \right]\kern-0.25em\right] _{1} \mathrm{d}t
\end{equation*}
\begin{equation*}
\hspace*{2.3cm}+\frac{\sqrt{b}\beta}{h}\int_0^T (\mathscr{S}^-V^n-V^n)_{N+1}^+ \left[\kern-0.25em\left[ {\mathscr{S}^-e^-_{V^n}} 
 \right]\kern-0.25em\right] _{N+1} \mathrm{d}t-\sqrt{b}\int_0^T (\mathscr{S}^+E-E)_{N+1}^- \left[\kern-0.25em\left[ {\mathscr{S}^-e^-_{V^n}} 
 \right]\kern-0.25em\right] _{N+1} \mathrm{d}t.
\end{equation*}
We know $ {{{({\mathscr{S}^ + }{e_E})}_\xi} \in \mathcal{P}^{k-1}}, \,{{{({\mathscr{S}^ - }{e_{V^n}})}_\xi}}  \in \mathcal{P}^{k-1},\,{{{({\mathscr{S}^ + }{e_L})}_\xi} \in \mathcal{P}^{k-1}}, \, {\mathscr{P}{e_R}} \in \mathcal{P}^{k} $,  Using projection properties     
\begin{equation*}
 \mathscr{S}^{\pm}:
\left( {\mathscr{S}^ + }L - L, {\mathscr{S}^ - }{e_{V^n}}_\xi \right)_{I_s} = 0, \left( {\mathscr{S}^ + }E - E, {\mathscr{S}^- e_{V^n}}_\xi \right)_{I_s} = 0,
\end{equation*}
\begin{equation*}
{\left( {{\mathscr{S}^ - }V^n - V^n,{{({\mathscr{S}^ + }{e_E})}_\xi}} \right)_{{I_s}}} = 0, \, \left( {\mathscr{S}R - R,{\mathscr{S}^ + }{e_E}} \right)_{{I_s}} = 0,\,
\end{equation*}
\begin{equation*}
\left( {\mathscr{S}R - R,({\mathscr{S}^ + }{e_E})_\xi} \right)_{I_s} = 0, ({\mathscr{S}^ + }E - E)_{s + 1} = 0, ({\mathscr{S}^ - }V^n - V^n)_{s + 1} = 0,
\end{equation*}
therefore 
\begin{equation*}
\hspace{-4cm}\mathscr{B}\left( {{\mathscr{S}^ - }V^n{\rm{ }} - V^n,{\mathscr{S}^ + }L{\rm{ }} - {\rm{ L}},{\mathscr{S}^ + }E - E,\mathscr{S}R - R;{\mathscr{S}^ - }{e_{V^n}},{\mathscr{S}^ + }{e_L}, - \mathscr{S}{e_R},{\mathscr{S}^ + }{e_E}} \right)
\end{equation*}
\begin{equation*}
\hspace*{-1.5cm} = \int_0^T \sum\limits_{s = 1}^N \left( \sum\limits_{j = 1}^M \Delta {\pi _j} W({\alpha _j}) \delta _t^{{\alpha _j}} (\mathscr{S^-} V^n - V^n), \mathscr{S}^- e_V \right)_{I_s} \mathrm{d}t + \int_0^T \sum\limits_{s = 1}^N \left( \mathscr{S}^+ L - L, \mathscr{S}^+ e_L \right)_{I_s} \mathrm{d}t
\end{equation*}
\begin{equation*}
\quad\quad\quad+\int_0^T \sum\limits_{s = 1}^N \left( {\mathcal{L}_{\frac{\alpha-2}{2}} (\mathscr{S}R - R)} - (\mathscr{S}^+E-E), \mathscr{S} e_R \right)_{I_s} \mathrm{d}t - \sqrt{b} \int_0^T (\mathscr{S}^+E-E^-)_{N+1} \left[\kern-0.25em\left[ {\mathscr{S}^-e_{V^n}} \right]\kern-0.25em\right]_{N+1}^- \mathrm{d}t.
\end{equation*}
Using Lemma \ref{lem20}  we have
\begin{equation*}
\left\| {\mathcal{L}_{\frac{\alpha-2}{2}} (\mathscr{S}R - R) - (\mathscr{S}^+E-E)} \right\| \le C h^{k + 1}.
\end{equation*}
 Combining this with Young’s inequality \cite{mitrinovic1970analytic} and  property \eqref{n4}, we obtain
\begin{equation*}
\hspace*{-6cm}\mathscr{B}\left( {{\mathscr{S}^ - }V{\rm{ }} - V,{\mathscr{S}^ + }L{\rm{ }} - {\rm{ L}},{\mathscr{S}^ + }E - E,\mathscr{S}R - R;{\mathscr{S}^ - }{e_V},{\mathscr{S}^ + }{e_L}, - \mathscr{S}{e_R},{\mathscr{S}^ + }{e_E}} \right)
\end{equation*}
\begin{equation*}
\le \int_0^T \sum\limits_{s = 1}^N \left( \sum\limits_{j = 1}^M \Delta {\pi _j} W({\alpha _j}) \delta _t^{{\alpha _j}} (\mathscr{S^-} V^n - V^n), \mathscr{S}^- e_V \right)_{I_s} \mathrm{d}t + C_{T,a,b} \left( h^{2k+2} + (\Delta t)^{4 + p} + p^4 \right)
\end{equation*}
\begin{equation*}
+ \frac{1}{C_{T,a,b}} \int _0^T \sum\limits_{s=1}^N\left\| {\mathscr{S}e_L} \right\|_{I_s}^2 \mathrm{d}t +\int_0^T \frac{\sqrt{b}\beta}{h}|(\mathscr{S^-}e_V)_{N+1}|^2 \mathrm{d}t + \int_0^T \sum\limits_{s = 1}^N {{\left\|{\mathscr{S}^ + }{e_L}\right\| }_{{I_s}}^2}\mathrm{d}t.
\end{equation*}
\end{proof}
\begin{thm}
Let  $V$ be a exact solution of the equation \eqref{eq1-1}  in $ \Omega \subset \mathbb{R} $
such that $ F(V) = 0 $. Assuming $ V^n_h $ is the numerical solution of the semi-discrete LDG scheme \eqref{s1}-\eqref{s5}.  For small enough $ h $, the error estimation  is as follows:
\begin{equation*}
{\left\| {V(\xi ,{t_n}) - V_h^n} \right\|_{{L^2}(\Omega )}} \le C\left( {{h^{k + 1}} + {{(\Delta t)}^{1 + \frac{p }{2}}} + {p ^2}} \right),
\end{equation*}
\end{thm}
\begin{proof}
Using Lemma \ref{Lemma3} with initial error $ \left\| {{\mathscr{S}^ - }{e_V}(0)} \right\| = 0 $ we have
\begin{equation*}
\hspace{-7cm}\mathscr{B}\left( {{\mathscr{S}^ - }e_V,{\mathscr{S}^ + }e_L,{\mathscr{S}^ + }e_E,\mathscr{S}e_R;{\mathscr{S}^ - }{e_V},{\mathscr{S}^ + }{e_L}, - \mathscr{S}{e_R},{\mathscr{S}^ + }{e_E}} \right)
\end{equation*}
\begin{equation*}
= \int_0^T \sum\limits_{s = 1}^N \left( \sum\limits_{j = 1}^M \Delta {\pi _j} W({\alpha _j}) \delta _t^{{\alpha _j}} \left(\mathscr{S^-}e_{V^n}\right), \mathscr{S}^- e_{V^n} \right)_{I_s} \mathrm{d}t + \int_0^T \sum\limits_{s = 1}^N \left( {\mathcal{L}_{\frac{\alpha-2}{2}} (\mathscr{S}e_R)}, \mathscr{S}e_R \right)_{I_s} \mathrm{d}t
\end{equation*}
\begin{equation*}
\hspace*{-3cm}+ \int_0^T \sum\limits_{s = 1}^N {{{\left\|{\mathscr{S}^+}{e_L}\right\|}_{{I_s }}^2}\mathrm{d}t }+\int_0^T \frac{\sqrt{b}\beta}{h}|(\mathscr{S^-}e_{V^n})_{N+1}|^2 \mathrm{d}t.
\end{equation*}
Recalling Lemma \ref{good}, we have
\begin{equation*}
\hspace*{-2cm}
\int_0^T \sum\limits_{s = 1}^N \left( \sum\limits_{j = 1}^M \Delta {\pi _j} W({\alpha _j}) \delta _t^{{\alpha _j}} \left(\mathscr{S^-}e_{V^n}\right), \mathscr{S}^- e_{V^n} \right)_{I_s} \mathrm{d}t + \int_0^T \sum\limits_{s = 1}^N \left( {\mathcal{L}_{\frac{\alpha-2}{2}} (\mathscr{S}e_R), \mathscr{S}e_R} \right)_{I_s} \mathrm{d}t
\end{equation*}
\begin{equation*}
\le \int_0^T \sum\limits_{s = 1}^N \left( \sum\limits_{j = 1}^M \Delta {\pi _j} W({\alpha _j}) \delta _t^{{\alpha _j}} \left(\mathscr{S^-}e_{V^n}\right), \mathscr{S}^- e_{V^n} \right)_{I_s} \mathrm{d}t + C_{T,a,b} h^{2k+2} + \frac{1}{C_{T,a,b}} \int_0^T \sum\limits_{s=1}^N \left\| {\mathscr{S}e_L} \right\|_{I_s}^2 \mathrm{d}t.
\end{equation*}
By using Lemma \ref{lem1-1} and property \eqref{n4} we have
\begin{equation} \label{n6}
\left\| \delta _t^\alpha \left( \mathscr{S}^+ V(\xi ,{t_n}) - V(\xi ,{t_n}) \right) \right\|_{L^2(\Omega )} \le C\left( h^{k + 1} + \Delta t^{2 - \alpha} \right).
\end{equation}
Using \eqref{n5}, \eqref{ta} and \eqref{n6} we have
\begin{equation}
{\left\| {\sum\limits_{j = 1}^M {W({\alpha _j})\Delta {\pi _j}\delta _t^{{\alpha _j}}\left( {{\mathscr S^ + }V(\xi ,{t_n}) - V(\xi ,{t_n})} \right)} } \right\|_{{L^2}(\Omega )}} \le C\left( {{h^{k + 1}} + (\Delta {t)^{1 + \frac{p }{2}}} + {p^2}} \right).
\end{equation}
Hence
\begin{equation*}
\left( \sum\limits_{j = 1}^M \Delta {\pi _j} W({\alpha _j}) \delta _t^{{\alpha _j}} \left(\mathscr{S^-}V^n - V^n\right), \mathscr{S}^- e_{V^n} \right) \le C\left( h^{2k + 2} + \Delta t^{2 + p} + p^4 \right) + c \left\| \mathscr{S}^- V^n - V^n \right\|_{L^2(\Omega)}^2.
\end{equation*}
It then follows that
\begin{equation*}
\hspace*{-2cm} \left( \sum\limits_{j = 1}^M \frac{W({\alpha _j})\Delta {\pi _j}}{\lambda _j} \left(\mathscr{S^-}V^n - V_h^n\right), \mathscr{S}^- e_{V^n} \right) \le \left( \sum\limits_{j = 1}^M \frac{W({\alpha _j})\Delta {\pi _j}}{\lambda _j} \sum\limits_{l = 1}^{n - 1} \left( a_{n - l - 1}^{{\alpha _j}} - a_{n - l}^{{\alpha _j}} \right) \left(\mathscr{S^-}V^l - V_h^n\right), \mathscr{S}^- e_{V^n} \right).
\end{equation*}
\begin{equation*}
\hspace*{3cm} + \left( \sum\limits_{j=1}^M \frac{W({\alpha _j})\Delta {\pi _j} a_{n - 1}^{{\alpha _j}} (\mathscr{S^-}V^0 - V_h^n)}{\lambda _j}, \mathscr{S}^- e_{V^n} \right) + c \left\| \mathscr{S}^- V^n - V^n \right\|_{L^2(\Omega)}^2
\end{equation*}
\begin{equation*}
\hspace*{0.5cm}+ C\left( {{h^{2k + 2}} + (\Delta {t)^{2 + p }} + {p ^4}} \right) 
\end{equation*}
By using Young’s inequality, we obtain
\begin{equation*}
\hspace*{-7cm}\left\| {{\mathscr S^ - }{e_{{V^n}}}} \right\|_{{L^2}(\Omega )}^2 \le \sum\limits_{j = 1}^M {  {\frac{{W({\alpha _j})\Delta {\pi _j}}}{{{\lambda _j}}}}\sum\limits_{l = 1}^{n - 1} {\left( {a_{n - l - 1}^{{\alpha _j}} - a_{n - l}^{{\alpha _j}}} \right)}}\left\| {{\mathscr S^ - }{e_{{V^l}}}} \right\|_{{L^2}(\Omega )}^2  
\end{equation*}
\begin{equation*}
\hspace*{2.5cm}+ \frac{1}{4} \sum\limits_{j = 1}^M {  {\frac{{W({\alpha _j})\Delta {\pi _j}}}{{{\lambda _j}}}}\sum\limits_{l = 1}^{n - 1} {\left( {a_{n  - 1}^{{\alpha _j}} - a_{n - 1}^{{\alpha _j}}} \right)}}\left\| {{\mathscr S^ - }{e_{{V^n}}}} \right\|_{{L^2}(\Omega )}^2 +\sum\limits_{j = 1}^M {  {\frac{{W({\alpha _j})\Delta {\pi _j}}}{{{\lambda _j}}}} { {a_{n  - 1}^{{\alpha _j}} } }}\left\| {{\mathscr S^ - }{e_{{V^0}}}} \right\|_{{L^2}(\Omega )}^2 
\end{equation*}
\begin{equation*}
\hspace*{2.5cm}+ \frac{1}{4}\sum\limits_{j = 1}^M {  {\frac{{W({\alpha _j})\Delta {\pi _j}}}{{{\lambda _j}}}} {\left( {a_{n  - 1}^{{\alpha _j}} } \right)}}\left\| {{\mathscr S^ - }{e_{{V^n}}}} \right\|_{{L^2}(\Omega )}^2  +cQ\left\| {{\mathscr S^ - }{e_{{V^n}}}} \right\|_{{L^2}(\Omega )}^2 + CQ\left( {{h^{2k + 2}} + (\Delta {t)^{2 + p }} + {p ^4}} \right).
\end{equation*}
Notice the facts that $ \left\| {{S^ - }{e_{{V^0}}}} \right\|_{{L^2}(\Omega )}^{} \le C{h^{k + 1}}. $ Thus,
\begin{equation*}
\hspace*{-7cm}\left\| {{\mathscr S^ - }{e_{{V^n}}}} \right\|_{{L^2}(\Omega )}^2 \le \sum\limits_{j = 1}^M {  {\frac{{W({\alpha _j})\Delta {\pi _j}}}{{{\lambda _j}}}}\sum\limits_{l = 1}^{n - 1} {\left( {a_{n - l - 1}^{{\alpha _j}} - a_{n - l}^{{\alpha _j}}} \right)}}\left\| {{\mathscr S^ - }{e_{{V^l}}}} \right\|_{{L^2}(\Omega )}^2  
\end{equation*}
\begin{equation*}
\hspace*{1cm}+(cQ+ \frac{1}{4}) \sum\limits_{j = 1}^M {  {\frac{{W({\alpha _j})\Delta {\pi _j}}}{{{\lambda _j}}}}\sum\limits_{l = 1}^{n - 1} {}}\left\| {{\mathscr S^ - }{e_{{V^n}}}} \right\|_{{L^2}(\Omega )}^2 +\sum\limits_{j = 1}^M {  {\frac{{W({\alpha _j})\Delta {\pi _j}}}{{{\lambda _j}}}} { {a_{n  - 1}^{{\alpha _j}} } }}h^{2k+2}
\end{equation*}
\begin{equation*}
\hspace*{2.5cm} + C\sum\limits_{j = 1}^M \left( {\frac{{W({\alpha j})\Delta {\pi j}}}{{{\lambda j}}}}\right) \sum\limits{l = 1}^{n - 1} \left| {{\mathscr S^ - }{e{{V^n}}}} \right|{{L^2}(\Omega )}^2 + C\sum\limits_{j = 1}^M \left( {\frac{{W({\alpha _j})\Delta {\pi _j}}}{{{\lambda j}}}} \right) a{n - 1}^{{\alpha _j}} \left( {{h^{2k + 2}} + (\Delta {t)^{2 + p }} + {p ^4}} \right)
\end{equation*}
Assuming that  $ C $ is very small such that $ \frac{3}{4}-cQ>0 $, we have
\begin{equation*}
\hspace*{-1cm}\left\| {{\mathscr S^ - }{e_{{V^n}}}} \right\|_{{L^2}(\Omega )}^2 \le C \left(\sum\limits_{j = 1}^M {  {\frac{{W({\alpha _j})\Delta {\pi _j}}}{{{\lambda _j}}}}\sum\limits_{l = 1}^{n - 1} {\left( {a_{n - l - 1}^{{\alpha _j}} - a_{n - l}^{{\alpha _j}}} \right)}}\left\| {{\mathscr S^ - }{e_{{V^l}}}} \right\|_{{L^2}(\Omega )}^2 +C\sum\limits_{j = 1}^M   {\frac{{W({\alpha _j})\Delta {\pi _j}}}{{{\lambda _j}}}}  {a_{n  - 1}^{{\alpha _j}} }\left( {{h^{2k + 2}} + (\Delta {t)^{2 + p }} + {p ^4}} \right) \right).
\end{equation*}
For $ n=1,2,3,...,m-1 $, we have
\begin{align*}
\hspace*{-1cm}\left\| {{\mathscr S^ - }{e_{{V^m}}}} \right\|_{{L^2}(\Omega )}^2 &\le C \left(\sum\limits_{j = 1}^M \left(  {\frac{{W({\alpha _j})\Delta {\pi _j}}}{{{\lambda _j}}}}\right)\sum\limits_{l = 1}^{m - 1} \left( {a_{n - l - 1}^{{\alpha _j}} - a_{n - l}^{{\alpha _j}}} \right)\right)\left( {{h^{2k + 2}} + (\Delta t)^{2 + p} + {p ^4}} \right) \\
&\quad + C\sum\limits_{j = 1}^M \left( {\frac{{W({\alpha _j})\Delta {\pi _j}}}{{{\lambda _j}}}} \right) a_{n  - 1}^{{\alpha _j}} \left( {{h^{2k + 2}} + (\Delta t)^{2 + p} + {p ^4}} \right) \\
\hspace*{-3cm}&= \left( {{h^{2k + 2}} + (\Delta t)^{2 + p} + {p ^4}} \right)
\end{align*}
then, by using  standard approximation theory we have
\begin{equation*}
{\left\| {V(\xi ,{t_m}) - V_h^m} \right\|_{{L^2}(\Omega )}} \le C\left( {{h^{k + 1}} + {{(\Delta t)}^{1 + \frac{p }{2}}} + {p ^2}} \right).
\end{equation*}
\end{proof}
\section{Numerical results}\label{Sec:5}
In this section, we solve three nonlinear numerical examples of the convection-diffusion equation of fractional order to demonstrate the accuracy and efficiency of the LDG method that is shown by employing L$^2$-error, $\displaystyle{{E}_h = \left\| {V - {V_h}} \right\|_2}$, and the approximate rate of convergence, $\displaystyle{{\Lambda _{order}}} =\frac{log(\text{E}_h)-log(\text{E}_{h/m})}{log(m)}$.
\begin{definition} \label{exam1}
Consider the following  time and space fractional nonlinear equation
\begin{equation*} 
 \frac{{\partial^\alpha V(\xi ,t)}}{{\partial t^\alpha}}+ \frac{{\partial }}{{\partial \xi}}\left( \frac{V^2(\xi,t)}{2}  \right) = \frac{{\partial }}{{\partial \xi}}{\left( \frac{{\partial V(\xi,t)}}{{\partial \xi}} \right)} + (-\mathcal{L})^{\frac{\beta}{2}} V(\xi,t)+ g(\xi,t),\quad   - 1 \le \xi  \le 1, \quad 0 < t \le 1,
\end{equation*}
\begin{equation*}
V_0(\xi)=0,
\end{equation*}
and 
\begin{equation*}
g(\xi,t) = \left( {{{\left( {{\xi ^2} - 1} \right)}^4}\frac{{{\partial ^\alpha }}}{{\partial {t^\alpha }}}{t^2} + 8{t^4}\xi {{({\xi ^2} - 1)}^7} + b{t^2}{{\left( { - \mathcal{L} } \right)}^{\frac{\beta }{2}}}{{\left( {{\xi ^2} - 1} \right)}^4}} \right).
\end{equation*}
The exact solution for $ \beta \in (1,2) $ is  $ V(\xi ,t) = {t ^2}{\left( {{\xi ^2} - 1} \right)^4} $ with $ b = \frac{{\Gamma (8 - \beta )}}{{\Gamma (8)}}. $ we take $ \Delta t = \frac{T}{{500}}, p= \frac{1}{50}. $
\begin{table}[ht!] 
\centering
{\footnotesize{
\caption{The comparison of the obtained  norm error and the convergence rate of LDG method with and without (\cite{aboelenen2018local}) Legendre polynomials for Example \ref{exam1}  versus $k$, $N$, and $\beta$ .} 
\setlength{\tabcolsep}{0.5em}
\centering 
\begin{tabular}{|c|c| c c a a|c c a a|} 
\hline
&&\multicolumn{4}{c|}{$k=1$}&\multicolumn{4}{c|}{$k=2$}\\
\cline{3-10}
&& \multicolumn{2}{c}{LDG method} & \multicolumn{2}{a|}{ method \cite{aboelenen2018local}} & \multicolumn{2}{c}{LDG method}& \multicolumn{2}{a|}{method \cite{aboelenen2018local}} \\
\hline
$\beta$ & $N$ & $E_h$ &$\Lambda _{order}$ & $E_h$ &$\Lambda _{order}$ & $E_h$ & $\Lambda _{order}$ & $E_h$ & $\Lambda _{order}$\\
\hline
\multirow{3}{*}{1.2} & 10 & 1.03e-03 & - & 1.23e-02 & - & 6.21e-04 & - & 8.35e-03 & -  \\
& 20 & 2.73e-04 & 1.91 & 4.61e-03 & 1.42 & 7.84e-05 & 2.98 & 1.21e-03 & 2.79 \\
& 40 & 6.84e-05 & 1.99 & 1.1e-03 & 2.03 & 9.36e-06 & 3.04 & 1.41e-04 & 3.21 \\
\hline
\multirow{3}{*}{1.4} & 10 & 1.23e-03 & - & 1.01e-02 & - & 5.31e-04 & - & 6.24e-03 & - \\
& 20 & 3.01e-04 & 2.03 & 2.51e-03 & 2.01 & 6.53e-05 & 2.96 & 9.23e-04 & 2.76  \\
& 40 & 7.45e-05 & 2.01 & 6.31e-04 & 1.96 & 8.05e-06 & 3.02 & 1.13e-04 & 3.14  \\ 
\hline
\multirow{3}{*}{1.8} & 10 & 6.21e-04 & - & 7.31e-03 & - & 4.22e-04 & - &2.62e-03 & -  \\
& 20 & 1.52e-04 & 2.03 & 1.91e-03 & 1.94 & 5.43e-05 & 2.95 & 3.54e-04 & 2.89  \\
& 40 & 3.78e-05 & 2.00 & 4.71e-04 & 1.99 & 6.76e-06 & 3.00 & 4.66e-05 & 3.08  \\ 
\hline
\end{tabular}
\label{Tbl:ex2}
}}
\end{table}

\end{definition}

\begin{definition} \label{ex12}
 Consider the following  problem:  
\begin{equation*} \label{ex2}
\frac{{{\partial ^\alpha }V(\xi ,t)}}{{\partial {t^\alpha }}} + \frac{{{\partial }  }}{{\partial {\xi }}}(\frac{V^4(\xi ,t)}{2}) +  \frac{{\partial V(\xi ,t)}}{{\partial \xi }}+ (-\mathcal{L})^{\frac{\beta}{2}} V(\xi,t) = g(\xi ,t),\,\,\,\,\,\,\,\,\,0 < \xi  < 1,\,\,\,0 < t \le 1,
\end{equation*}
with the initial condition
\begin{equation*}
V(\xi ,0) = 0,\,\,\,\,\,0 < \xi  < 1,
\end{equation*}
and the boundary conditions
\begin{equation*}
V(0,t) = t^3,\,\,\,\,\, V(1,t) = 0,\,\,\,\,0 < t \le 1,
\end{equation*}
where
\begin{equation*}
g(\xi ,t) = \left[ {{{\left( {1 - {\xi^2}} \right)}^2}\frac{{{\partial ^\alpha }}}{{\partial {t^\alpha }}}{t^2} - 4{t^9}\xi{{\left( {1 - {\xi^2}} \right)}^7} + {t^3}\left( { - 4 + 12{\xi^2}} \right) + b{t^3}{{\left( { - \mathcal{L}} \right)}^{\frac{\beta }{2}}}{{(1 - {\xi^2})}^2}} \right].
\end{equation*}
The exact solution of \ref{ex2} is $ V(\xi ,t) = {t^3}\left( {1 - {\xi }^2} \right)^2 $ with $ b=\frac{\Gamma(8-\beta)}{\Gamma (8)} $. 
\begin{table}[ht!] 
\centering
{\footnotesize{\caption{The LDG method for various $ \beta $ and $ k $ when $ T=1 $, $ \Delta t = \frac{T}{{500}}, p= \frac{1}{50} $ for example \ref{ex12}.} \label{table3}
\setlength{\tabcolsep}{0.5em}
\centering 
\begin{tabular}{|c|c| c c|c c |c c|} 
\hline
&&\multicolumn{2}{c|}{$k=1$}&\multicolumn{2}{c|}{$k=2$}&\multicolumn{2}{c|}{$k=3$}\\
\hline
$\beta$ & $N$ & $E_h$ &$\Lambda _{order}$ & $E_h$ & $\Lambda _{order}$ & $E_h$ & $\Lambda _{order}$\\
\hline
\multirow{3}{*}{1.2} & 10 & 1.45e-04 & - & 1.75e-04 & - & 2.45e-05& -\\
& 20 & 3.55e-05 & 2.03 & 2.25e-05 & 2.95 & 1.50e-06 & 4.02 \\
& 40 & 8.63e-06 & 2.04 & 2.91e-06 & 2.96 & 9.50e-08 & 3.98\\
& 80 & 2.17e-06 & 1.99 & 3.55e-07 & 3.03 & 5.83e-09 & 4.02\\
\hline
\multirow{3}{*}{1.6} & 10 & 1.34e-04 & - & 2.22e-05 & - & 2.43e-05 &-\\
& 20 & 3.28e-05 & 2.03 & 2.81e-06 & 2.98 & 1.53e-06 & 3.98\\
& 40 & 8.33e-06 & 1.97 & 3.55e-07 & 2.98 & 9.60-08 & 3.99 \\ & 80 &   2.09e-06 & 1.98 & 4.55e-08 & 2.96 & 5.97e-09 & 4.00\\ 
\hline
\multirow{3}{*}{1.8} & 10 & 1.22e-04 & - & 2.12e-05 & - & 4.54e-05 & -\\
& 20 & 3.09e-05 & 1.98 & 2.67e-06 & 2.98 & 2.86e-06 & 3.98\\
& 40 & 7.75e-06 & 1.99 & 3.27e-07 & 3.02 & 1.79e-07 & 3.99\\ 
& 80 &  1.92e-06 & 2.01 & 4.10e-08 & 2.99 & 1.11e-08 & 4.01\\ 
\hline
\end{tabular}
}}
\end{table}
\end{definition}

\begin{definition}
\label{exam3}
Consider the following problem:
\begin{equation} \label{eq31}
\begin{cases}
\frac{{\partial^ \alpha V(\xi,t)}}{{\partial t^ \alpha}} + \frac{\partial }{{\partial \xi}}\left( {\frac{{{V^2}(\xi,t)}}{2}} \right) = b\left( { - {{\left( { - \mathcal{L} } \right)}^{\frac{\beta}{2}}}} \right)V(\xi,t) + Z(\xi,t), \quad  (\xi,t) \in  [-2,2] \times (0,0.5],\\
V(\xi,0) = {V_0}(\xi), \quad   \xi \in [-2,2],
\end{cases} 
\end{equation}
with the discontinuous initial condition
\begin{equation*}
{V_0}(\xi) = \begin{cases}
\frac{{{{(1 - {\xi^2})}^4}}}{{10}}, & - 1 \le \xi \le 1,\\
    0, & \text{otherwise}.
\end{cases} 
\end{equation*}
In this example, we set $ b=1 $ and consider the source term as
\begin{equation*}
Z(\xi,t) ={V_0}(\xi)\frac{\partial ^\alpha }{\partial {t^\alpha }}{e^{-t}} +{e^{ - t}}\left(   {e^{ - t}}{V_0}(\xi){V'_0}(\xi) + {\left( { - \mathcal{L} } \right)^{\frac{\beta }{2}}}{V_0}(\xi)\right).
\end{equation*}
The exact solution is
\begin{equation*}
V(\xi,t) = \begin{cases}
\frac{{{{e^{-t}(1 - {\xi^2})}^4}}}{{10}},\quad  - 1 \le \xi \le 1,\\
0, \, \quad  \quad  \quad  \quad \quad \text{otherwise}.
\end{cases}
\end{equation*}

\begin{table}[ht!] 
\centering
{\footnotesize{\caption{Error and temporal convergence orders for various $ \beta $ and $ \Delta t $ when $ T=0.5, $ for example \ref{exam3}.} \label{table4}
\setlength{\tabcolsep}{0.5em}
\centering 
\begin{tabular}{|c| c c|c c |c c|} 
\hline 
 $\beta$ &\multicolumn{2}{c|}{$\beta=1.2$}&\multicolumn{2}{c|}{$\beta=1.6$}&\multicolumn{2}{c|}{$\beta=1.8$}\\
\hline 
 $\Delta t$ & $E_h$ &$\Lambda _{order}$ & $E_h$ & $\Lambda _{order}$ & $E_h$ & $\Lambda _{order}$\\
\hline
 $ T/100 $ & 3.33e-04 & - & 1.30e-04 & - & 1.02e-04& -\\
$ T/200 $ & 1.65e-04 & 1.01 & 6.43e-05 & 1.01 & 4.94e-05 & 1.04 \\
 $ T/400 $ & 0.81e-04 & 1.02 & 3.14e-05 & 1.03 & 2.42e-05 & 1.02 \\
 $ T/800 $ & 0.40e-04 & 1.01 & 1.51e-05 & 1.05 & 1.17e-05 & 1.04\\
\hline
\end{tabular}
}}
\end{table}
\begin{table}[ht!] 
\centering
{\footnotesize{\caption{Error and numerical integration convergence orders for various $ \beta $ and $ \Delta t $ at $ T=0.5, $ when $ p $ is small enough for example \ref{exam3}.} \label{table5}
\setlength{\tabcolsep}{0.5em}
\vspace*{-.5cm}
\centering 
\begin{tabular}{|c| c c|c c |c c|} 
\hline 
 $\beta$ &\multicolumn{2}{c|}{$\beta=1.3$}&\multicolumn{2}{c|}{$\beta=1.7$}&\multicolumn{2}{c|}{$\beta=1.8$}\\
\hline 
 $p$ & $E_h$ &$\Lambda _{order}$ & $E_h$ & $\Lambda _{order}$ & $E_h$ & $\Lambda _{order}$\\
\hline
 $ 1/10 $ & 3.14e-04 & - & 3.44e-04 & - & 2.31e-04& -\\
$ 1/20 $ & 7.53e-05 & 2.06 & 8.53e-05 & 2.01 & 5.63e-05 & 2.03 \\
 $ 1/40 $ &1.83e-05 & 2.04 &2.10e-05 & 2.03 & 1.38e-05 & 2.02 \\
 $ 1/80 $ & 4.48e-06 & 2.03 & 5.15e-06 & 2.02 & 3.39e-06 & 2.03\\
\hline
\end{tabular}
}}
\end{table}

\end{definition}

\section*{Conclusions}  
This study employs the local discontinuous Galerkin (LDG) method with Legendre polynomial basis functions to approximate non-linear convection-diffusion governed by space and time fractional Laplacian operators. We recast the principal problem into a first-order system before leveraging the discontinuous Galerkin approach. Our findings indicate that the method's precision can be enhanced through judicious selection of basis functions. Specifically, using the Legendre basis function, we establish that the proposed LDG technique is stable and exhibits convergence of $O(h^{k+1}+(\Delta t)^{1+\frac{p}{2}}+p^2)$.
Our computational results corroborate this analysis, highlighting the superiority of these polynomials over conventional ones for methodological basis. Additionally, we deduce that the method's accuracy scales positively with the degree of the basis function.

\section*{Conflict of Interest}
The authors declare no conﬂict of interest.

\bibliographystyle{ieeetr}

\bibliography{reference}

\end{document}